\documentclass[a4paper,11pt]{scrartcl}

\usepackage{graphicx}
\usepackage[utf8]{inputenc}
\usepackage[english]{babel}
\usepackage{amsmath,amssymb,enumerate,url,tabularx,stmaryrd,mathrsfs}

\usepackage{todonotes,comment,afterpage}
\usepackage{subcaption}
\usepackage{framed}
\usepackage{mathtools}
\usepackage{fullpage}
\usepackage[title]{appendix}
\usepackage{nicefrac}

\usepackage{amsthm}
\usepackage{graphicx}
\usepackage{hyperref}

\theoremstyle{plain}
\newtheorem{theorem}{Theorem}[section]
\newtheorem{lemma}[theorem]{Lemma}
\newtheorem{corollary}[theorem]{Corollary}
\newtheorem{proposition}[theorem]{Proposition}

\newtheorem{remark}[theorem]{Remark}

\numberwithin{equation}{section}
\numberwithin{figure}{section}
\numberwithin{table}{section}

\newcommand{\R}{\mathbb{R}}

\newcommand{\N}{\mathbb{N}}

\newcommand{\abs}[1]{\left|#1\right|}
\newcommand{\norm}[1]{\left\|#1\right\|}

\newcommand{\support}{\operatorname{supp}}

\def\UU{\mathcal{U}}
\def\VV{\mathcal{V}}
\def\YY{\mathcal{Y}}
\def\WW{\mathcal{W}}

\def\fdiv{\operatorname{div}}

\def\LL{\mathcal{L}}

\newcommand{\extension}{\mathcal{E}}

\def\HHwr{H^1_{\rho}(y^\alpha,\R^d\times\R^+)}
\def\HHwry{H^1_{\rho}(y^\alpha,\R^d \times (0,\YY))}

\def\HH{\mathbb{H}}

\def\trace{\mathrm{tr_0}}

\usepackage{tikz}
\usetikzlibrary{shapes}
\usepackage{pgfplots}
\usetikzlibrary{calc}

\usetikzlibrary{positioning}
\usetikzlibrary{arrows}
\usetikzlibrary{calc}
\usetikzlibrary{intersections, pgfplots.fillbetween,patterns}
\usepgfplotslibrary{colorbrewer}

\iftrue 

\usetikzlibrary{external}
\tikzexternaldisable

\else

\fi

\title{Fractional Diffusion in the full space: decay and regularity}

\author{Markus Faustmann\footnote{
Institute for Analysis and Scientific Computing, TU Wien, Vienna, Austria,
\texttt{markus.faustmann@tuwien.ac.at}} 
\, and Alexander Rieder\footnote{
Institute for Analysis and Scientific Computing, TU Wien, Vienna, Austria,
\texttt{alexander.rieder@tuwien.ac.at}}}
\date{\today}

\begin{document}
\maketitle
\begin{abstract}
We consider fractional partial differential equations posed on the full space $\R^d$. Using the well-known Caffarelli-Silvestre extension to $\R^d \times \R^+$ as equivalent definition, we derive existence and uniqueness of weak solutions. We show that solutions to a truncated extension problem on $\R^d \times (0,\YY)$ converge to the solution of the original problem as $\YY \rightarrow \infty$. Moreover, we also provide an algebraic rate of decay and derive weighted analytic-type regularity estimates for solutions to the truncated problem.
 These results pave the way for a rigorous analysis of numerical methods for the full space problem, such as FEM-BEM coupling techniques.  
 
\end{abstract}

\section{Introduction}
In recent years, models using non-integer powers of differential operators garnered lots of interest as the inherent non-locality of these operators gives a more accurate way to describe non-local processes in physics, finance or image processing,  \cite{bv16,collection_of_applications}.
Restricting these non-local PDE models to some bounded domain requires one to fix values of the solution everywhere outside of the domain, which may lead to some non-physical assumptions for the boundary conditions.  Consequently, the full-space problem is oftentimes used in analytical works.

In a similar vein, when working on a bounded domain, there are multiple non-equivalent definitions of fractional differential operators such as the fractional Laplacian, \cite{whatis}. The most common ones are the integral fractional Laplacian (defined pointwise as a singular integral) and the spectral fractional Laplacian (defined using spectral calculus). Consequently, 
 it is oftentimes not obvious, which definition of the fractional Laplacian should be used in the model. In contrast, working on the full space, one obtains a single natural definition as  all different approaches are equivalent, \cite{tendef}.

In this work, we analyze fractional PDEs in the full space. Using the  influential interpretation of elliptic fractional differential operators as Dirichlet-to-Neumann operators for degenerate elliptic PDEs, the so called Caffarelli-Silvestre extension, \cite{caffarelli_silvestre,stinga_torrea}, defined on the half space $\R^d \times \R^+$, we show well-posedness of a weak formulation of the fractional PDE. As, in general, analytical solutions to such problems are unknown, discretizations of the equations are usually employed to derive approximative solutions.  

In the case of fractional PDEs on bounded domains $\Omega \subset \R^d$, see e.g. \cite{NOS15,tensor_fem}, a truncation to $\Omega \times (0,\YY)$ is used to be able to discretize the extension problem. This induces two natural questions: does the solution to the truncated extension problem (with homogeneous Neumann condition on the artificial boundary) converge to the solution of the original problem and can the rate of convergence be quantified? 
For bounded domains, \cite{tensor_fem} answered both questions by showing exponential decay in $\YY$ by exploiting an explicit representation for the $y$-dependence. 

In this article, we employ the truncation to the full space problem, i.e., we study the extension problem on $\R^d \times (0,\YY)$ and answer both questions as well. In this case, however, 
there is no closed form expression for the $y$-dependence available. Nonetheless, we show convergence of the truncated solution to the original solution in the full-space setting, but only with certain algebraic rates. From a technical standpoint, the explicit representation is replaced by applying purely variational techniques to show the decay properties.

\subsection{Impact on numerical methods}

Numerical methods for fractional PDEs on bounded domains are fairly developed, as can e.g. be seen in the survey articles
\cite{bbnos18,EDGGTZ20,whatis} and we especially mention approximations based on the finite element method (FEM), \cite{ab17,tensor_fem,abh19,fmk22}. 
A key limitation to the FEM is the restriction to bounded computational domains. A classical reformulation for exterior problems uses boundary integral equations, which leads to the boundary element method (BEM), \cite{sauter_schwab}.
An approach for transmission problems on unbounded domains that is commonly employed is the combination of both methods, so called FEM-BEM couplings, \cite{costabel88a,han90}.
The goal of our follow-up work, \cite{FR23}, is to formulate a fully computable symmetric FEM-BEM coupling method applied to fractional transmission problems posed in $\R^d$.

However, before a rigorous analysis of any numerical method can be made, analytical foundations regarding well-posedness and regularity of the problem at hand must be made. 
As a second key result of this article, we establish analyticity of the solution in the extended direction $y$ in terms of certain weighted Sobolev spaces. This is achieved by deriving a small initial regularity shift in a weighted space and then employing bootstrapping arguments to control higher-order derivatives. Structurally, these estimates are similar to the ones for the case of bounded domains in  \cite{tensor_fem,fmms21} and show that solutions are in certain countably normed spaces.

 Combined with our follow-up work \cite{FR23}, this article establishes that the Caffarelli-Silvestre extension approach
  can be combined with FEM-BEM coupling techniques to yield a good approximation
  scheme.

\subsection{Layout}

The present paper is structured as follows: In Section~2, we introduce our model problem and formulate assumptions on the data to be able to apply FEM-BEM techniques afterwards. Then, the Caffarelli-Silvestre extension as well as its weak formulation and the weak formulation of the truncated problem are introduced. Finally, we present our main results: Proposition~\ref{prop:discrete_well_posedness} shows well-posedness of both weak formulations, Proposition~\ref{lemma:truncation_convergence} provides convergence of the truncated solution to the solution posed on $\R^d \times \R^+$ and Proposition~\ref{lemma:truncation_error} gives the algebraic rate of decay. In Proposition~\ref{prop:regularity} the regularity results
in weighted Sobolev spaces are presented.

Section~3 is then devoted to the proofs of the well-posedness and convergence results, where the key step is Lemma~\ref{lemma:stronger_decay}, which shows decay properties of the full space solution as the truncation parameter $\YY \rightarrow \infty$ by employing inf-sup theory and weighted spaces. 

Finally, in Section~4 the estimates for higher order derivatives are derived. Hereby, an initial regularity shift in Lemma~\ref{lemma:initial_shift} and Lemma~\ref{lemma:apriori_decay_l2} allows to use an induction argument to show Proposition~\ref{prop:regularity}. Moreover, a (finite) regularity result in the non-extended variables and a characterization of the solution in certain countably normed spaces is presented.

\subsection{Notations}
Throughout the text we use the symbol $a \lesssim b$ meaning that $a \leq Cb$ with a generic constant $C>0$ that is independent of any crucial quantities in the analysis. Moreover, we write $\simeq$ to indicate that both estimates $\lesssim$ and $\gtrsim$ hold.

For any multi index $\alpha = (\alpha_1,\dots,\alpha_d)\in \mathbb{N}^d_0$, we
denote the partial derivative $\partial^\alpha = \partial^{\alpha_1}_{x_1}\cdots \partial^{\alpha_d}_{x_d}$ of order
$|\alpha|= \sum_{i=1}^d\alpha_i$.
Moreover, for $k \in \N$, we employ classical integer order Sobolev spaces $H^k(\Omega)$ on (bounded) Lipschitz domains $\Omega$ and the fractional Sobolev spaces $H^t(\R^d)$ for $t \in (0,1)$ defined, e.g., via Fourier transformation.

\section{Main results}
\subsection{Model problem}
We consider a stationary fractional diffusion problem on the full space $\R^d$ with $d=2$ or $d=3$ given by
\begin{align}\label{eq:stat_model_prob}
  \LL^{\beta} u + s u &= f \;  \qquad \text{in $\R^d$} 
\end{align}
with $s \geq 0$,  and $\beta \in (0,1)$.
The self-adjoint operator $\LL$ is hereby defined as
\begin{align*}
  \LL u := -\fdiv\big( \mathfrak{A} \nabla u\big),
\end{align*}
and, for functions $u \in L^2(\R^d)$, the fractional differential operator $\LL^{\beta}$ is defined using spectral calculus
$$
\LL^{\beta} u:=\int_{\sigma(\LL)} z^\beta dE \; u,
$$
where $E$ is the spectral measure of $\LL$ and $\sigma(\LL)$ is the spectrum of $\LL$.
Using standard techniques this definition can be extended to tempered distributions.

For the data, we assume 
that $\mathfrak{A}: \R^d \rightarrow \R^{d\times d}$ is smooth and pointwise symmetric and positive definite in the sense that there exists $\mathfrak{A}_0 >0$ such that
\begin{align*}
(\mathfrak{A}(x) y,y)_2 \geq \mathfrak{A}_0 \norm{y}_2^2 \qquad \forall y \in \R^d.
\end{align*}
In order to avoid several additional difficulties due to decay conditions at infinity, we assume $s \geq\sigma_0>0$ for the case $d=2$. 
\bigskip

Additionally, we make the following  assumptions on the coefficients in the model problem: There exists a bounded Lipschitz domain $\Omega \subseteq \R^d$ such that
\begin{enumerate}
\item $\support f \subseteq \Omega$,
\item  $\mathfrak{A} \equiv I$ in $\R^d\setminus \overline{\Omega}$.
\end{enumerate}

\begin{remark}
 We note that adding lower order terms to the operator is also covered by our techniques, i.e., 
 \begin{align*}
  \LL u := -\fdiv\big( \mathfrak{A} \nabla u\big) + \mathfrak{c} u,
\end{align*}
where $\mathfrak{c}: \R^d \rightarrow \R$ with $\mathfrak{c} \geq 0$ is smooth and satisfies $\mathfrak{c} \equiv \mathfrak{c}_0 \in \R$
  in $\R^d\setminus \overline{\Omega}$.
  However, in order to make the key concepts more clear, we decided to stick to the case $\mathfrak{c} = 0$ in the following.
\end{remark}

\subsection{The Caffarelli-Silvestre extension}
Following~\cite{stinga_torrea}, we rewrite~\eqref{eq:stat_model_prob} as an extension problem in a half space in $\R^{d+1}$. The extension problem is conveniently described using weighted Sobolev spaces. 

For any bounded open subset
$D\subset \R^d\times \R$ and arbitrary $\alpha \in (-1,1)$, we define $L^2(y^\alpha,D)$ as the space of square integrable functions with respect to the weight $y^\alpha$. Correspondingly, the Sobolev space   $H^1(y^\alpha,D) \subset L^2(y^\alpha,D)$ consists of functions, for which the norm
$$
\norm{\UU}_{H^1(y^\alpha,D)}^2 := \int\int_{D}y^\alpha\Big( \big|\nabla \UU(x,y)\big|^2 
  +\big| \UU(x,y)\big|^2 \Big)\,dx\,dy
$$ 
is finite. 

As our model problem is formulated on an \emph{unbounded} domain, we need to take care of the behaviour at infinity.
To that end, we use appropriately
  weighted Sobolev spaces, as is standard for the Poisson problem, see e.g. \cite{AGG94}.
For $(x,y) \in \R^d\times \R$, we introduce the weight
\begin{align*}
\rho(x,y):=(1+\abs{x}^2+\abs{y}^2)^{1/2}. 
\end{align*}
For a (possibly unbounded) domain $D \subset \R^d \times \R^+$, we define the 
space $H^1_{\rho}(y^\alpha,D)$ as the space of all square integrable functions $\UU$  (with respect to the weight function $y^\alpha \rho^{-2}$) such that the norm
\begin{align}\label{eq:normHHw}
  \norm{\UU}_{H^1_{\rho}(y^\alpha,D)}^2 :=
  \int\int_{D}{y^\alpha \Big( \big|\nabla \UU(x,y)\big|^2
  + \rho(x,y)^{-2} \big| \UU(x,y)\big|^2 \Big)\,dx\,dy}
\end{align}
is finite. 
Commonly used cases are $D=\R^d\times \R^+$ (full space), $D = \R^d \times (0,\YY)$ for $\YY>0$ (corresponding to truncation in $y$-direction), or $D=\omega \times (0,\YY)$ 
for $\omega \subset \R^d$ and $\YY >0$.
%

\begin{remark}
 For bounded sets $\omega \subset \R^d$ and $\YY < \infty$, we sometimes use the weighted spaces $H^1_{\rho}(y^\alpha, \omega\times (0,\YY))$, noting that, in this case, the weight satisfies $1 \leq \rho(x,y) \leq C(\omega,\YY) < \infty$. Consequently, the norm \eqref{eq:normHHw} defines an equivalent norm to the $H^1(y^\alpha,\omega\times (0,\YY))$-norm.
\end{remark}

For functions $\UU \in \HHwr$, one can give meaning to their trace at $y=0$, which we denote by $\operatorname*{tr}_0\UU$. In fact, Lemma~\ref{lem:trace} will show that $\operatorname*{tr}_0\UU$ is in a weighted fractional Sobolev space.

Then, the extension problem reads as: find $\UU \in \HHwr$ such that
  \begin{subequations} \label{eq:extension}
\begin{align}
 - \fdiv\big(y^\alpha \mathfrak{A}_x \nabla \UU\big) &= 0 \qquad \text{in $\R^{d} \times \R^+$},\\
  d^{-1}_{\beta}\partial_{\nu^\alpha} \UU +s \trace\UU&= f \qquad \text{in $\R^{d}$},
\end{align}
  \end{subequations} 
where $d_{\beta}:=2^{1-2\beta}\Gamma(1-\beta)/\Gamma(\beta)$, $\alpha:=1-2\beta \in (-1,1)$, $\partial_{\nu^\alpha} \UU(x) := -\lim_{y\rightarrow 0} y^\alpha \partial_y \UU(x,y)$, and $\mathfrak{A}_x = \begin{pmatrix} \mathfrak{A} & 0 \\ 0 & 1 \end{pmatrix} \in \R^{(d+1) \times (d+1)}$.
Then, by \cite{stinga_torrea}, the solution to \eqref{eq:stat_model_prob} is given by $u = \UU(\cdot,0)$.
\bigskip

The weak formulation of \eqref{eq:extension} in $\HHwr$ reads as finding $\UU \in \HHwr$ such that
\begin{align}\label{eq:weakform}
  A(\UU,\VV)&:=
  \int_{0}^{\infty}{y^\alpha \int_{\R^d}\mathfrak{A}_x(x)\nabla \UU \cdot \nabla \VV \; dx dy} + s d_{\beta} \int_{\R^d}{\trace \UU \trace \VV\,dx} = d_{\beta}(f,\trace \VV)_{L^2(\R^d)}
  \end{align}
  for all $\VV \in \HHwr$. 
    If $s>0$, it is natural to include the trace term into the norm.
    Thus, we introduce:
    \begin{align*}
      \|\UU\|^2_{\HH}:= \|\UU\|^2_{\HHwr} + s \|\trace \UU\|^2_{L^2(\R^d)}.
    \end{align*}
  The first step towards a computable formulation, before even considering any
  discretization steps, is to cut the problem from the infinite cylinder $\R^d \times \R^+$ to a finite
cylinder in the $y$-direction. To do so, we fix a parameter $\YY >0$ to be chosen later and
introduce the truncated bilinear form
\begin{align*}
    {A}^{\YY}(\UU,\VV)&:=
  \int_{0}^{\YY}{y^\alpha \int_{\R^d}{\nabla \UU \cdot \nabla \VV} \;dx dy} + s d_{\beta} \int_{\R^d}{\trace \UU \trace \VV\,dx}.
\end{align*}

The truncated problem then reads: Find
$ \UU^{\YY} \in H^1_\rho(y^\alpha,\R^d\times (0,\YY))$  such that
\begin{align}\label{eq:truncated_BLF_eq}
  A^{\YY}(\UU^\YY,\VV^\YY)
  =d_\beta\big(f,\trace{\VV^\YY}\big)_{L^2(\R^d)} \quad 
  \text{for all } \VV^\YY \in H^1_\rho(y^\alpha,\R^d\times (0,\YY)).
\end{align}
In the following, we will often take $\YY \in (0,\infty]$ and refer to solutions to problem \eqref{eq:truncated_BLF_eq}, meaning that in the case $\YY = \infty$ these functions actually satisfy \eqref{eq:weakform}.

We also introduce a natural norm on the truncated cylinder:
\begin{align*}
  \|\UU\|^2_{\HH^\YY}:= \|\UU\|^2_{\HHwry} + s \|\trace \UU\|_{L^2(\R^d)}^2.
\end{align*}

 In fact, the truncated problem \eqref{eq:truncated_BLF_eq} corresponds to a weak formulation of a Caffarelli-Silvestre extension problem with an additional Neumann boundary condition at $y=\YY$: 
   \begin{subequations}    \label{eq:truncated_problem_pw}
  \begin{align}
    -\fdiv\big(y^\alpha \mathfrak{A}_x \nabla {\UU}^\YY\big) &= 0  &&\text{in $\R^{d} \times (0,\YY)$},\\
    d_{\beta}^{-1}\partial_{\nu^\alpha} {\UU}^\YY + s \trace{\UU}^\YY&=f &&\text{on $\R^{d} \times \{0\}$} , \\
    \partial_{y} {\UU}^\YY&=0 &&\text{on $\R^{d} \times \{\YY\}$}.
  \end{align}
  \end{subequations} 

\subsection{Main results}
We are now in position to formulate the main results of the article. The proofs of the statements are relegated to the following Sections ~\ref{sec:well-posedness} and \ref{sec:regularity}.
\subsubsection{Well-posedness and decay}
Regarding well-posedness of our variational formulation, we have the following proposition.
\begin{proposition}
  \label{prop:discrete_well_posedness}
  Assume that either $d>2$ or $s>0$.
  Then, problem \eqref{eq:weakform} has a unique solution $\UU \in \HHwr$ and there is a constant $C>0$ such that
  \begin{align*}
    \norm{\UU}_{\HH} &\leq C \min(1, s^{-1}) \norm{f}_{L^2(\Omega)}.
  \end{align*} 
  Fix $\YY \in (0,\infty)$. Then, the truncated problem \eqref{eq:truncated_BLF_eq} has a unique solution $\UU^{\YY}\in \HHwry$ satisfying
    \begin{align*}
    \norm{\UU^{\YY}}_{\HH_{\YY}} &\leq C \left(1+\frac{1}{\YY}\right) \min(1, s^{-1}) \norm{f}_{L^2(\Omega)}
  \end{align*}
with a constant $C>0$ independent of $\YY$.

  Moreover, the bilinear forms in~\eqref{eq:weakform} and \eqref{eq:truncated_BLF_eq} are coercive.
\end{proposition}

By the following proposition, we also obtain that solutions to the truncated problem converge to solutions to the non-truncated problem as the truncation parameter $\YY$ tends to infinity.

\begin{proposition}
  \label{lemma:truncation_convergence}
  Let $\UU$ solve~\eqref{eq:weakform} and, for $\YY >0$, let
  $\UU^\YY$ solve~\eqref{eq:truncated_BLF_eq}. For any fixed
  $0<\widetilde{\YY} < \YY$,
  it holds that  
  $\UU^\YY \to \UU$ in $H^1_{\rho}(y^\alpha, \R^d \times (0,\widetilde{\YY}))$
  as $\YY \to \infty$.
  If $s>0$, there additionally holds $\trace \UU^\YY \rightarrow \trace \UU$ in $L^2(\R^d)$ as $\YY \to \infty$.
\end{proposition}

Finally, we also obtain algebraic rates of convergence as $\YY\rightarrow \infty$ for the difference of the truncated and the non-truncated full-space solutions.

\begin{proposition}
  \label{lemma:truncation_error}
  Fix $\YY >0$. Let $\UU$ solve~\eqref{eq:weakform} and $\UU^\YY$ solve~\eqref{eq:truncated_BLF_eq}.   
   Let $\mu$ be given by
  \begin{align}
    \label{eq:def_of_decay_mu}
    \mu:=
    \begin{cases}
      1+\abs{\alpha} & s>0 \\
       1+\alpha & s = 0
    \end{cases}.
  \end{align}
  Then, there exists a constant $C>0$ depending only on $\alpha$ and $d$ such that 
  \begin{align*}
    \|\UU^{\YY} - \UU\|^2_{\HHwry}
    +s\|\trace(\UU^{\YY} - \UU)\|^2_{L^2(\R^d)}
    \leq C \YY^{-\mu}   
    \norm{f}^2_{L^2(\Omega)}.
  \end{align*}
\end{proposition}

\subsubsection{Regularity}

For solutions to the extension problem as well as the truncated extension problem there hold analytic type weighted estimates for the extended variable.  Estimates of that type allow to employ $hp$-finite elements in the extended variable, which will be considered in \cite{FR23}.

\begin{proposition}[Regularity in $y$]
  \label{prop:regularity}
  Fix $\YY \in (0,\infty]$ and let
  $\ell \in \N$. Let $\UU$ solve \eqref{eq:truncated_BLF_eq}. Then,  
  there exists constants $C,K>0$
  and $\varepsilon \in (0,1)$ such that the following estimate holds:
  \begin{align*}
    \big\|{y^{\ell-\varepsilon}\nabla\partial^{\ell}_y \UU}\big\|_{L^2(y^\alpha,\R^d\times(0,\YY))}
    \leq CK^{\ell} \ell! \norm{f}_{L^2(\Omega)}.
  \end{align*}
  All constants are independent of $\ell, \YY,$ and $\UU$.
\end{proposition}

In fact, the regularity results imply that solutions to our model problem are in certain countably normed spaces.
Following \cite[Sec.~5.5.1]{tensor_fem}, we introduce the Bochner spaces $L^2_\alpha((0,\infty);X)$ of square integrable functions (with respect to the weight $y^\alpha$) and values in the Banach space $X$ as well as for constants $C,K>0$, the countably normed spaces 
\begin{align*}
  \mathcal{B}^1_{\varepsilon,0}(C,K;X) := \Big\{ \VV \in C^\infty((0,\infty);X) : & 
\norm{\VV}_{L^2(y^\alpha,(0,\infty);X)} < C, \\
&\norm{y^{\ell+1-\varepsilon}\VV^{(\ell+1)}}_{L^2(y^\alpha,(0,\infty);X)} < C K^{\ell+1} (\ell+1)! \; \forall \ell \in \N_0
\Big\}.
\end{align*}

Proposition~\ref{prop:regularity} provides control of $y^{\ell-\varepsilon} \partial_y^{\ell+1} \UU$, which directly gives the following Corollary.

\begin{corollary}\label{cor:regularity_y}
Fix $\YY \in (0,\infty]$ and let $\UU$ solve \eqref{eq:truncated_BLF_eq}.
  Then, there are constants $C,K>0$ such that there holds 
  \begin{align}
    \partial_{y}\UU \in \mathcal{B}^1_{\varepsilon,0}(C,K;L^2(\R^d)).
  \end{align}
\end{corollary}

We note that we formulated the previous corollary in terms of $\partial_y \UU$, whereas the regularity results in \cite[eqn.~(6.10)]{tensor_fem} are formulated for solutions $\UU$ to the extension problem on bounded domains. This is due to the fact that in the case of the full space problem the estimates do not hold for the lowest order term as $\UU \notin L^2(\R^{d}\times \R^+)$. Nonetheless, the regularity result of Corollary~\ref{cor:regularity_y} (together with $\UU \in H^1_\rho(\R^{d}\times \R^+)$) allows to construct interpolation operators in a similar way as in \cite[Lem.~11]{tensor_fem}.
\bigskip

Finally, we investigate the regularity in $x$. Since this will depend
on the regularity  of the data $\mathfrak{A}$ and $f$, we only
consider the case of finite regularity.
\begin{proposition}[Regularity in $x$]
\label{prop:regularity_x}
Assume that $\mathfrak{A} \in C^{m}(\R^d;\R^{d\times d})$ and
$f \in H^{m}(\Omega)$. Then, for every multiindex $\zeta \in \N_0^d$ with $\abs{\zeta} = m$ there holds
\begin{align*}
  \|\nabla\partial_x^\zeta\UU\|_{L^2(y^\alpha,\R^d\times \R^+)} &\leq C \|f\|_{H^{m}(\Omega)}.
\end{align*}
The constant $C$ depends on $\Omega, \mathfrak{A}$, $m$ and $d$, but is independent of $f$ and $\UU$.
\end{proposition}

\section{Well-posedness and decay}
\label{sec:well-posedness}
In this section, we provide the proofs of Proposition~\ref{prop:discrete_well_posedness} (well-posedness), Proposition~\ref{lemma:truncation_convergence} (convergence) and Proposition~\ref{lemma:truncation_error} (algebraic rate of decay). 

\subsection{Trace estimate}
We start with a trace estimate in a certain weighted Sobolev space.

  \begin{lemma}\label{lem:trace}
    For all $\UU \in \HHwr$, there holds 
    \begin{subequations}
      \label{eq:trace}
      \begin{align}
        \label{eq:trace_semi}
      \abs{\trace \UU}_{H^\beta(\R^d)}
      \leq C \norm{\nabla \UU}_{L^2(y^\alpha,\R^d\times\R^+)}.
    \end{align}
    For $d=3$, we additionally have
    \begin{align}
      \label{eq:trace_l2}
      \|(1+|x|^2)^{-\beta/2} \trace\UU\|_{L^2(\R^d)}
      \leq C \norm{\nabla \UU}_{L^2(y^\alpha,\R^d\times\R^+)}.
    \end{align}
   In both cases the constant $C>0$ does only depend on $d$ and $\alpha$.
  \end{subequations}
\end{lemma}

\begin{proof}
  The estimate \eqref{eq:trace_semi} is shown in~\cite[Lem.~3.8]{KarMel19}.
  To estimate the weighted $L^2$-norm, we use interpolation space theory. 

More precisely,  \cite[Lemma 23.1]{tartar06} shows that interpolation of $L^2$-spaces with weights $w_0$ and $w_1$ denoted by $L^2(w_i,\R^d)$ for $i=0,1$ produces an interpolation space (using the K-method) $[L^2(w_0,\R^d),L^2(w_1,\R^d)]_{\theta,2} = L^2(w_\theta,\R^d)$ that is a weighted $L^2$-space with weight $w_\theta = w_0^{1-\theta}w_1^\theta$. 
Applying this result with $\theta = 1-\beta$ and $w_0=\rho_x^{-2} := \rho(x,0)^{-2} = (1+|x|^2)^{-1}$ and $w_1 = 1$, shows that 
  \begin{align*}
    \|\rho_x^{-\beta} \trace \UU\|^2_{L^2(\R^d)} = \|\trace \UU\|^2_{L^2(\rho_x^{-2\beta},\R^d)}
    &\lesssim \|\trace \UU\|^2_{[L^2(\rho_x^{-2},\R^d),L^2(\R^d)]_{1-\beta,2}}. 
  \end{align*}
  Now, by \cite[Lemma 40.1]{tartar06} the interpolation spaces can be seen as trace spaces, i.e., elements of the interpolation space can be seen as traces (at 0) of functions $\UU(y)$
satisfying $y^{1-\beta}\norm{\UU(y)}_{L^2(\rho_x^{-2},\R^d)} \in L^2(y^{-1},\R^+)$ as well as $y^{1-\beta}\norm{\partial_y \UU(y)}_{L^2(\R^d)} \in L^2(y^{-1},\R^+)$. Together with $\alpha=1-2\beta$ and the  Poincar\'e estimate from \cite[Theorem 3.3]{AGG94} (using the assumption $d=3$), this leads to
  \begin{align*}
\|\trace \UU\|^2_{[L^2(\rho_x^{-2},\R^d),L^2(\R^d)]_{1-\beta,2}} 
    &\lesssim \int_{0}^{\infty}{y^{\alpha} \|\rho_x^{-1} \UU(y)\|^2_{L^2(\R^d)} \,dy}
    +\int_{0}^{\infty}{y^{\alpha} \|\partial_y \UU(y)\|^2_{L^2(\R^d)} \,dy}\\
    &\lesssim \|\nabla \UU\|^2_{L^2(y^\alpha,\R^d \times \R^+)},
  \end{align*}
which produces the desired estimate.
\end{proof}

\subsection{Poincar\'e inequalities and well-posedness}
We now show the well-posedness of our variational formulations. The main ingredient is a Poincar\'e type estimate.

 \begin{lemma}
    \label{lemma:my_poincare}
Let $\alpha \in (-1,1).$ Let $\YY\in (0,\infty]$ and $\UU \in H^1_\rho(y^\alpha,\R^d \times (0,\YY))$. There exists a $\mu_0 >0$ such that for all $\mu \in [0,\mu_0)$ there holds
   \begin{align}
    \label{eq:my_poincare}
\int_{0}^{\YY}\int_{\R^d}{y^{\alpha}\rho^{\mu-2}|\UU|^2 \,dx} dy
    &\leq C\left(
 \int_{0}^{\YY}\int_{\R^d}{y^{\alpha}\rho^{\mu}|\nabla\UU|^2 \,dx} dy + |3-d| \|\trace\UU\|_{L^2(\R^d)}^2\right)
  \end{align}
  provided the right-hand side is finite.
\end{lemma}
\begin{proof}
For Poincar\'e inequalities on the full-space without the additional weight $y^\alpha$, we refer to \cite{AGG94}. 
Estimate \eqref{eq:my_poincare} for the case $d=3$ follows directly from multiplying a full-space Poincar\'e-inequality, see for example~\cite[Theorem 3.3]{AGG94}, applied only in $x$ with $y^\alpha$ and integrating over $(0,\YY)$. More details can also be found in our forthcoming work \cite{FR23}.

It remains to show \eqref{eq:my_poincare} for $d=2$. 
We write $\UU(x,y) = \UU(x,0) + \int_0^y \partial_y \UU(x,\tau) \; d\tau$, which gives
\begin{align*}
\int_0^\YY\int_{\R^d}y^\alpha \rho^{\mu-2} |\UU|^2 \; dxdy \lesssim \int_0^\YY\int_{\R^d}y^\alpha \rho^{\mu-2} |\UU(x,0)|^2 + y^\alpha \rho^{\mu-2}\Big( \int_0^y \partial_y \UU(x,\tau) \; d\tau \Big)^2 \; dxdy.
\end{align*} 
Since $\int_0^\YY y^\alpha \rho^{\mu-2} \lesssim 1$ for sufficiently small $\mu < \mu_0$, with $\mu_0>0$ depending only on $\alpha$, the first term on the left-hand side can be bounded by $C \norm{\trace \UU}_{L^2(\R^d)}^2$. For the second term, we employ a weighted Hardy-inequality, see e.g. \cite{muckenhoupt72}, to obtain
\begin{align*}
 \int_0^\YY\int_{\R^d}y^\alpha \rho^{\mu-2}\Big( \int_0^y \partial_y \UU(x,\tau) \; d\tau \Big)^2 \; dxdy \lesssim \int_{\R^d} \int_0^\YY y^\alpha \rho^\mu |\partial_y \UU|^2 \; dy dx,
\end{align*}
which shows the claimed inequality.
\end{proof}

Using this Poincar\'e- type inequality, we can now look at the well-posedness of our problem.

\begin{proof}[Proof of Proposition~\ref{prop:discrete_well_posedness}]
 The boundedness of the bilinear forms $\mathcal{A}(\cdot,\cdot)$ and $\mathcal{A}^\YY(\cdot,\cdot)$ follows directly from the Cauchy-Schwarz inequality and the definition of the norms 
$ \norm{\cdot}_{\HH}$ and $ \norm{\cdot}_{\HH_\YY}$ respectively.

Let $\YY \in (0,\infty]$.
  Coercivity of the bilinear forms follows directly from the Poincar\'e inequalities in Lemma~\ref{lemma:my_poincare}, since  
  \begin{align*}
  \norm{\UU^\YY}_{\HH_\YY}^2 &=  \int_{0}^{\YY}\int_{\R^d}{y^{\alpha}\rho^{-2}|\UU^\YY|^2 \,dx} dy+ \int_{0}^{\YY}\int_{\R^d}{y^{\alpha}|\nabla\UU^\YY|^2 \,dx} dy + s \norm{\trace \UU^\YY}_{L^2(\R^d)}^2 \\
  &\stackrel{\eqref{eq:my_poincare}}{\lesssim}  \int_{0}^{\YY}\int_{\R^d}{y^{\alpha}|\nabla\UU^\YY|^2 \,dx} dy + (s+(3-d)) \norm{\trace \UU^\YY}_{L^2(\R^d)}^2.
  \end{align*}
By assumption on $s$ and $d$, the trace term is not present for the case $s=0$. Therefore,  the right-hand side can be bounded by $C \mathcal{A}^\YY(\UU^\YY,\UU^\YY)$. 
  
  Thus, the Lax-Milgram lemma shows well-posedness provided the right-hand side of the variational formulation is a bounded linear functional.
 For the case $s > 0$, we can directly use the definition of the $\HH_\YY$-norm together with $\operatorname*{supp} f \subset \Omega$ to obtain
  \begin{align*}
    \int_{\R^d} f \trace{\UU^\YY} \; dx
    &\leq s^{-1}\norm{f}_{L^2(\Omega)} s\norm{\trace{\UU^\YY}}_{L^2(\R^d)}
      \leq s^{-1}\norm{f}_{L^2(\Omega)} \norm{\UU^\YY}_{\HH_\YY}.
  \end{align*}
  For $\YY = \infty$ and $s=0$, which implies $d=3$ by assumption, the trace estimate \eqref{eq:trace_l2} gives 
  \begin{align*}
    \int_{\R^d} f \trace{\UU} \; dx
    &\leq \norm{\rho(x,0)^{\beta}f}_{L^2(\Omega)} \norm{\rho(x,0)^{-\beta}\trace{\UU}}_{L^2(\R^d)}
\lesssim \norm{f}_{L^2(\Omega)} \norm{\nabla \UU}_{L^2(y^\alpha,\R^d \times \R^+)} \\
     & \leq \norm{f}_{L^2(\Omega)} \norm{\UU}_{\HH}.
  \end{align*}
For the case $\YY < \infty$ and $s=0$, we use a cut-off function $\chi$ satisfying $\chi \equiv 1$ on $(0,\YY/2)$, $\operatorname*{supp}\chi \subset (0,\YY)$ and $\norm{\nabla \chi}_{L^\infty(\R^+)}\lesssim \YY^{-1}$. As $\Omega$ is bounded, this gives with the trace estimate \cite[Lem.~3.7]{KarMel19}
  \begin{align*}
    \int_{\R^d} f \trace{\UU^\YY} \; dx
    &\leq \norm{f}_{L^2(\Omega)} \norm{\trace{(\chi\UU^\YY)}}_{L^2(\Omega)} \\
&\lesssim \norm{f}_{L^2(\Omega)} \left(\norm{\chi \UU^\YY}_{L^2(y^\alpha,\Omega\times (0,\YY))} +\norm{\nabla(\chi \UU^\YY)}_{L^2(y^\alpha,\Omega \times (0,\YY))} \right) \\
&\lesssim \norm{f}_{L^2(\Omega)} \left(\norm{\UU^\YY}_{L^2(y^\alpha,\Omega\times (0,\YY))} +\frac{1}{\YY}\norm{\nabla \UU^\YY}_{L^2(y^\alpha,\Omega \times (0,\YY))} \right) \\
     & \leq C \left(1+\frac{1}{\YY}\right)\norm{f}_{L^2(\Omega)} \norm{\UU^\YY}_{\HH_\YY},
  \end{align*}
  which finishes the proof.
\end{proof}

\subsection{The truncation error}

In the following subsection, we study the truncated problem \eqref{eq:truncated_BLF_eq}. The main goal is to derive decay estimates in the truncation parameter $\YY$ and consequently convergence of the solution of the truncated problem to the solution of the non-truncated problem as $\YY \rightarrow \infty$.
\bigskip

The following lemma is the key to the main results of Proposition~\ref{lemma:truncation_convergence} and Proposition~\ref{lemma:truncation_error}. 
Using inf-sup theory we obtain that solutions to the Caffarelli-Silvestre extension problem and the truncated problem (in $y$-direction) lie in certain weighted Sobolev spaces. The additional weights then directly provide the rates of decay. In fact, we establish that the solutions are in two different types of weighted spaces: spaces weighted with $(1+y)^\mu$ with $\mu$ given by \eqref{eq:def_of_decay_mu} (decay only in $y$) and spaces with weights $\rho^\varepsilon$ for sufficiently small $\varepsilon$ (decay in all directions).

\begin{lemma}
  \label{lemma:stronger_decay}
  Let $y_0>0$. Fix $\YY\in (y_0,\infty)$, and let $\mu$ be given by \eqref{eq:def_of_decay_mu}.
  Let $\UU^\YY$ solve~\eqref{eq:truncated_BLF_eq}.
  Then, $\UU^\YY$ satisfies the estimate
  \begin{align}
    \label{eq:decay_in_y}
    \int_{0}^{\YY}{y^{\alpha}\Big[(1+y)^\mu \|\nabla \UU^\YY(y)\|_{L^2(\R^d)}^2\!
    \!+\!(1+y)^\mu\!\|\rho(\cdot,y)^{-1} \UU^\YY(y)\|_{L^2(\R^d)}^2\Big]\,dy}
    \leq C \min(s^{-1},1)^2\norm{f}_{L^2(\Omega)}^2.
  \end{align}
  In addition, for $\YY \in (0,\infty]$, there exists $\varepsilon>0$, depending only on
  $\alpha$ and $\Omega$ such that
  \begin{align}
    \label{eq:decay_all_directions}
    \int_{0}^{\YY}{y^{\alpha} \int_{\R^d} \rho^{\varepsilon} |\nabla \UU^\YY(x,y)|^2 dx dy}
    \leq C \min(s^{-1},1)^2\norm{f}_{L^2(\Omega)}^2.
  \end{align}
  In both cases, the constant $C$ does only depend on $\Omega, d, \alpha,$ and $y_0$.
\end{lemma}
\begin{proof}
  By the uniqueness of Proposition~\ref{prop:discrete_well_posedness}, it suffices to show existence
  of such a solution.  To that end, we use $\inf$-$\sup$-theory, see, e.g., \cite[Thm.~2.1.44]{sauter_schwab}, i.e., we have to show 
  \begin{align*}
  \inf_{\UU \in \mathcal{X}_{\mu,\YY} \backslash \{0\}} \sup_{\VV \in \mathcal{Y}_{-\mu,\YY} \backslash \{0\}} \frac{\abs{A^{\YY}(\UU,\VV)}}{\norm{\UU}_{\mathcal{X}_{\mu,\YY}}\norm{\VV}_{\mathcal{Y}_{-\mu,\YY}}} \geq \gamma >0 \qquad \text{(inf-sup condition)}, \\
 \forall \VV \in \mathcal{Y}_{-\mu,\YY} \backslash \{0\} \; : \;  \sup_{\UU \in \mathcal{X}_{\mu,\YY} \backslash \{0\}} \abs{A^{\YY}(\UU,\VV)} > 0 \qquad \text{(non-degeneracy condition)}
  \end{align*}
with spaces  $\mathcal{X}_{\mu,\YY}$, $\mathcal{Y}_{-\mu,\YY}$ 
  specified in the following. 
  
  We define the ansatz space
  $\mathcal{X}_{\mu,\YY}$ as a subspace of $H_{\rho}^1(y^\alpha,\R^d\times (0,\YY))$ of functions for which the norm
  $$
  \|\UU\|_{\mathcal{X}_{\mu,\YY}}^2:= \int_{0}^{\YY}{y^{\alpha}(1+y)^\mu \|\nabla \UU(y)\|_{L^2(\R^d)}^2 \, dy }
    + s\norm{\trace\UU}_{L^2(\R^d)}^2 
     $$
  is finite.  
  \bigskip

\textbf{Step 1 (Proof of~\eqref{eq:decay_in_y} with $\mu=1-\alpha$):}  We start with the simpler case $s>0$ and take $\mu = 1-\alpha$. Let $\chi(y):=\begin{cases}1 \quad &y \leq 1 \\ y^{1-\alpha} \quad &y >1 \end{cases}$.
 For $\UU\in \mathcal{X}_{\mu,\YY}$, we define
  $\VV:=(1+\delta \chi(y))\UU $ (for some $0<\delta<1$ to be fixed later) and calculate 
  \begin{align*}
  \int_0^\YY   \int_{\R^d} {y^\alpha \mathfrak{A}_x \nabla \UU \cdot \nabla \VV dxdy}
    &\geq
    \mathfrak{A}_0 \int_{0}^{\YY}{y^{\alpha}(1+\delta \chi(y)) \|\nabla \UU(y)\|_{L^2(\R^d)}^2 dy} \\
 & \qquad   +\int_{1}^{\YY}\int_{\R^d}{y^\alpha \delta (1-\alpha) y^{-\alpha}\UU \partial_y \UU \; dx dy} \\   
    &=
      \mathfrak{A}_0 \int_{0}^{\YY}{y^{\alpha}(1+\delta \chi(y)) \|\nabla \UU(y)\|_{L^2(\R^d)}^2 dy} \\
      &\qquad +
      \frac{\delta (1-\alpha)}{2}\int_{\R^d}\int_{1}^{\YY}
      { \frac{\partial }{\partial_y}\big(\UU^2\big) \; dy dx}\\
    &=
      \mathfrak{A}_0 \int_{0}^{\YY}{y^{\alpha}(1+\delta \chi(y)) \|\nabla \UU(y)\|_{L^2(\R^d)}^2 dy} \\
      & \qquad -  \frac{\delta (1-\alpha)}{2}\int_{\R^d}{\UU(x,1)^2\,dx} 
      +\frac{\delta (1-\alpha)}{2}\int_{\R^d}{\UU(x,\YY)^2\,dx} 
    \\
    &\geq \mathfrak{A}_0 \int_{0}^{\YY}{y^{\alpha}(1+\delta \chi(y)) \|\nabla \UU(y)\|_{L^2(\R^d)}^2 dy} -  \frac{\delta (1-\alpha)}{2}\int_{\R^d}{\UU(x,1)^2\,dx}. 
  \end{align*}
  In order to estimate the last term, we employ
  \begin{align*}
  \UU(1)^2 &\leq 2\UU(0)^2 + 2\abs{\int_0^1\partial_y \UU(y) \; dy}^2 \leq 2\UU(0)^2 + 2\int_0^1 y^\alpha |\partial_y \UU(y)|^2 dy \int_0^1 y^{-\alpha} dy \\
  &= 2\UU(0)^2 + \frac{2}{1-\alpha}\int_0^1 y^\alpha |\partial_y \UU(y)|^2 dy,
  \end{align*}
  which gives   using $1+\delta \chi(y) \geq \frac{\delta}{4} (1+y)^{1-\alpha}$
  \begin{align*}
   \int_0^\YY   \int_{\R^d} {y^\alpha \mathfrak{A}_x \nabla \UU \cdot \nabla \VV \; dx dy} &\geq (\mathfrak{A}_0 -\delta) \int_{0}^{\YY}{y^{\alpha}(1+\delta \chi(y)) \|\nabla \UU(y)\|_{L^2(\R^d)}^2 dy} \\ 
   &\qquad-  \delta (1-\alpha)\int_{\R^d}{\UU(x,0)^2\,dx}
   \\
   &\geq \frac{\delta}{4}(\mathfrak{A}_0 -\delta) \int_{0}^{\YY}{y^{\alpha}(1+y)^{1-\alpha} \|\nabla \UU(y)\|_{L^2(\R^d)}^2 dy} \\
   &\qquad-  \delta (1-\alpha)\int_{\R^d}{\UU(x,0)^2\,dx}.
  \end{align*}
  Consequently, we obtain
  \begin{align}\label{eq:decay_tmp1}
    A^\YY(\UU,\VV)&\geq \frac{\delta}{4}(\mathfrak{A}_0-\delta)\int_{0}^{\YY}{y^{\alpha}(1+y)^{1-\alpha}\|\nabla \UU(y)\|_{L^2(\R^d)}^2\; dy} \nonumber   \\
    &\qquad   + \Big(s d_\beta-\delta(1-\alpha)\Big)\|\trace \UU\|^2_{L^2(\R^d)}.
  \end{align}
  Choosing $\delta < \min(\mathfrak{A}_0/2, s d_\beta/(2-2\alpha))$, both terms on the right-hand side in \eqref{eq:decay_tmp1} are non-negative and using
   $\|\mathcal{V}\|_{\mathcal{X}_{\alpha-1,\YY}}\lesssim \|\mathcal{U}\|_{\mathcal{X}_{1-\alpha,\YY}}$, which follows easily from $(1+\delta \chi(y)) \lesssim (1+y)^{1-\alpha}$, gives the inf-sup condition for the ansatz space $\mathcal{X}_{1-\alpha,\YY}$ and the test space $\mathcal{X}_{\alpha-1,\YY}$. Moreover, the inf-sup constant behaves like $\sim \min(1, s)$.
  
  The non-degeneracy condition follows essentially with the same arguments, as, for given $\mathcal{V}$, the function $\mathcal{U} := (1+\delta \chi(y))^{-1} \mathcal{V}$ provides the positivity of the bilinear form.
  
The definition of the norm in the test-space and $\support f \subset \Omega$ implies 
$$(f,\trace \VV)_{L^2(\R^d)} \leq \norm{f}_{L^2(\Omega)} \norm{\VV}_{\mathcal{X}_{\alpha-1,\YY}},$$
which gives a bound for the right-hand side. Now, general inf-sup theory provides the existence of a solution that satisfies the claimed decay properties.
  \bigskip

  \textbf{Step 2 (Proof of~\eqref{eq:decay_in_y} with $\mu=1+\alpha$):}
  Next, we show that the rate of decay $\mu=1+\alpha$ is possible for $s>0$ and even for $s=0$. In the following, we only discuss the harder case $s=0$ as for $s>0$, we only obtain an additional non-negative term in the bilinear form.
 Here, we use the test space induced by the norm 
  \begin{align*}
    \|\VV\|_{\widetilde{\mathcal{Y}}_{-\mu,\YY}}^2
    &:=
      \int_{0}^{\YY}{\frac{y^{\alpha}}{\ln(y+2)^2(1+y)^{\mu}}\|\nabla \VV(y)\|_{L^2(\R^d)}^2 \,dy}
      + \norm{\trace{\VV}}_{L^2(\Omega)}^2.
  \end{align*}
 For given $\UU \in \mathcal{X}_{\mu,\YY}$, we choose the test function 
  $$
  \VV(x,y):= y^{1+\alpha}\UU(x,y) + (1+\alpha)\int_{y}^{\YY}{\tau^\alpha \UU(x,\tau) \,d\tau}
  $$
 with the derivatives
  \begin{align*}
    \nabla_x \VV &= y^{1+\alpha}\nabla_x \UU + (1+\alpha)\int_{y}^{\YY}{\tau^\alpha \nabla_x\UU(\tau)\,d\tau}
                   \quad \text{and}\quad
  \partial_y \VV(y)= y^{1+\alpha}\partial_y \UU(y).    
  \end{align*}
    The function $\VV$ is indeed in the test space, since we can bound the norm $\|\VV\|_{\widetilde{\mathcal{Y}}_{-1-\alpha},\YY}$ by
    \begin{align}\label{eq:decaytmp1}
      \|\VV\|_{\widetilde{\mathcal{Y}}_{-1-\alpha},\YY}^2
      &=
        \int_{0}^{\YY}{\frac{y^{\alpha}}{\ln(y+2)^2(1+y)^{1+\alpha}}\|\nabla \VV(y)\|_{L^2(\R^d)}^2 \,dy}
        + \norm{\trace{\VV}}_{L^2(\Omega)}^2 \nonumber \\
      &\lesssim
        \int_{0}^{\YY}{\frac{y^{\alpha+2(1+\alpha)}}{\ln(y+2)^2(1+y)^{1+\alpha}}\|\nabla \UU(y)\|_{L^2(\R^d)}^2
        \,dy} \nonumber
      \\
      &\quad+\int_{\R^d}{\int_{0}^{\YY}\frac{y^{\alpha}}{\ln(y+2)^2(1+y)^{1+\alpha}}
       \Big |\int_{y}^{\YY}{\tau^{\alpha} \nabla_x\UU(\tau)\,d\tau}\Big|^2 dy dx}
        + \norm{\trace{\VV}}_{L^2(\Omega)}^2.
    \end{align}
    Since the first term is readily bounded due to $\UU \in \mathcal{X}_{1+\alpha,\YY}$ and $1+\alpha>0$, we focus on the second. Using a weighted Hardy inequality, see e.g. \cite{muckenhoupt72}, with the weight $y^{-1/2}/\ln(y+2)$ that is square integrable in $\R^+$ we obtain
    \begin{align}\label{eq:decaytmp2}
     & \int_{\R^d}{\int_{0}^{\YY}\frac{y^{\alpha}}{\ln(y+2)^2(1+y)^{1+\alpha}}
      \Big|\int_{y}^{\YY}{\tau^{\alpha} \nabla_x\UU(\tau)\,d\tau}\Big|^2 dy dx} \nonumber \\
      &\qquad\qquad\qquad\qquad\leq
        \int_{\R^d}{\int_{0}^{\YY}{
        \Big|\frac{y^{-1/2}}{\ln(y+2)}\int_{y}^{\YY}{\tau^{\alpha} \nabla_x\UU(\tau)\,d\tau}\Big|^2 dy dx}} \nonumber\\
      &\qquad\qquad\qquad\qquad\lesssim
        \int_{\R^d}{\int_{0}^{\YY}{y^{1+2\alpha}
        |\nabla_x\UU(y)|^2}dy dx} \leq \|\UU\|_{\mathcal{X}_{1+\alpha,\YY}}^2.
    \end{align}
    What is left is to bound the trace of $\VV$. 
We use a cut-off function $\chi$ satisfying $\chi \equiv 1$ on $(0,y_0/2)$, $\support \chi \subset (0,y_0)$, and $\norm{\nabla \chi}_{L^\infty(\R)} \leq C$ with a constant $C$ depending only on $y_0$. Then,
\begin{align*}
\VV(x,0)^2 &= (\chi \VV)(x,0)^2 = \Big(\int_0^{y_0} \partial_y (\chi \VV)(x,y) \; dy \Big)^2
\leq \frac{y_0^{1-\alpha}}{1-\alpha} \int_0^{y_0} y^\alpha  |\partial_y (\chi \VV)|^2 \; dy \\
&\lesssim \int_0^{y_0} y^\alpha \left( |\partial_y \VV|^2 + |\partial_y \chi|^2\VV^2 \right) \; dy. 
\end{align*}
Integration over $\Omega$ and using the definition of $\VV$ gives 
\begin{align*}
 \norm{\trace{\VV}}_{L^2(\Omega)}^2 &\lesssim \int_0^{y_0} y^\alpha  \|\nabla \VV(y)\|_{L^2(\Omega)}^2 dy \\ 
 &\qquad+  \int_\Omega\int_0^{y_0} y^{2+3\alpha} |\partial_y \chi|^2\UU^2 dy dx + \int_\Omega\int_0^{y_0} y^{\alpha}\Big|\int_{y}^{\YY}{\tau^\alpha \UU(x,\tau) \,d\tau}\Big|^2 dy dx. 
\end{align*}
On $\Omega \times (0,y_0)$ we can insert any appearing weights in the ansatz-space and test-space as needed, which just adds multiplicative constants independent of $\YY$. Moreover, we can employ standard Poincar\'e-inequalities to bound the $L^2$-norm (here, the integrand even vanishes on $(0,y_0/2)$). 
Repeating the arguments from \eqref{eq:decaytmp1} and  \eqref{eq:decaytmp2} (with slightly changed weight in the Hardy inequality to insert the weight $\rho^{-2}$), we obtain the bound 
\begin{align*}
 \norm{\trace{\VV}}_{L^2(\Omega)} \lesssim \|\UU\|_{\mathcal{X}_{1+\alpha},\YY}.
\end{align*}
    Thus, we have shown 
    $\|\VV\|_{\widetilde{\mathcal{Y}}_{-1-\alpha,\YY}}\lesssim \|\UU\|_{\mathcal{X}_{1+\alpha},\YY}$.
We continue with inserting $\UU,\VV$ into the truncated bilinear form $A^\YY(\cdot,\cdot)$, which leads to
  \begin{align}
    \label{eq:decay_estimate_proof1}
  A^\YY(\UU,\VV) = \int_0^\YY \int_{\R^d}{y^\alpha \mathfrak{A}_x \nabla \UU \cdot \nabla \VV \; dx dy} 
    &\geq
    \mathfrak{A}_0 \int_{0}^{\YY}{y^{1+2\alpha} \|\nabla \UU(y)\|_{L^2(\R^d)}^2 dy} \nonumber
    \\ &\quad  +(1+\alpha)\int_{\R^d}\int_{0}^{\YY}{y^{\alpha} \mathfrak{A}^{1/2}\nabla_x \UU
      \int_{y}^{\YY}{\tau^{\alpha} \mathfrak{A}^{1/2}\nabla_x\UU(\tau)\,d\tau} \,dy\,dx} \nonumber\\
    &=: I + II .
  \end{align}
  We show that the term $II$ is non-negative. To simplify notation, we
  write $v(y):=\mathfrak{A}^{1/2}\nabla_x \UU(y)$ and suppress the $x$-dependency.
  We note that by the chain rule there holds
  $$
  y^\alpha v(y) \cdot \int_{y}^{\YY}{\tau^\alpha v(\tau)\,d\tau}
  = - \frac{1}{2}\frac{d}{d y} \Big|\int_{y}^{\YY}{\tau^{\alpha} v(\tau)\,d\tau}\Big|^2.
    $$
    This gives for the second term in \eqref{eq:decay_estimate_proof1}:
    \begin{align*}
      II
    &=
      -\frac{(1+\alpha)}{2}\int_{\R^d}{\int_{0}^{\YY}{
      \frac{d}{d y} \Big|\int_{y}^{\YY}{\tau^{\alpha} v(\tau)\,d\tau}\Big|^2\,dy} dx} \\
    &=
      \frac{(1+\alpha)}{2}\int_{\R^d}{\Big|      
      \int_{0}^{\YY}{\tau^{\alpha} v(\tau)\,d\tau}\Big|^2 dx}
      \geq 0. 
    \end{align*}
    Overall, we get using $(1+y^{1+\alpha})\gtrsim (1+y)^{1+\alpha}$ 
    \begin{align*}
    A^\YY(\UU,\VV) + A^{\YY}(\UU,\UU) &\geq \mathfrak{A}_0 \int_{0}^{\YY}{y^{\alpha}(1+y^{1+\alpha}) \|\nabla \UU(y)\|_{L^2(\R^d)}^2 dy} \gtrsim\|\UU\|_{\mathcal{X}_{1+\alpha},\YY}^2 
    \\ &
    \gtrsim\|\UU\|_{\mathcal{X}_{1+\alpha},\YY}\|\UU+\VV\|_{\widetilde{\mathcal{Y}}_{-1-\alpha},\YY},
    \end{align*}
    where the last inequality follows from the triangle inequality and  $\|\VV\|_{\widetilde{\mathcal{Y}}_{-1-\alpha,\YY}}\lesssim \|\UU\|_{\mathcal{X}_{1+\alpha},\YY}$.

   For the non-degeneracy condition, for a given $\VV$, we can choose $\UU=\VV$, which is in the ansatz-space, since due to $\YY<\infty$ the weights in the gradient terms in the ansatz- and test-space are equivalent.  
   
       By definition of the test-space and $\operatorname{supp} f \subset \Omega$, there holds $(f,\trace \VV)_{L^2(\R^d)} \leq \norm{f}_{L^2(\Omega)} \norm{\VV}_{\widetilde{\mathcal{Y}}_{-1-\alpha},\YY}$. Consequently, we obtain unique solvability of our weak formulation in the ansatz-space, which gives the decay estimate. 
    \bigskip

    \textbf{Step 3 (Proof of~\eqref{eq:decay_all_directions}):}
     Again, we use inf-sup theory with a different ansatz space. Here, for $\varepsilon>0$, we choose it to be a subspace of $H^1_\rho(y^\alpha,\R^d \times \R^+)$ such that additionally $\int_{\R^d \times (0,\YY)}{y^\alpha \rho^{\varepsilon} \abs{\nabla \UU}^2 }$ is finite. We only work out the case $s=0$ in the following, for $s>0$, the same argument can be made by additionally including a trace term in the norm.
    Setting $z:=(x,y) \in \R^{d+1}$ and $\VV(z):= \rho^{\varepsilon}(z) \UU(z)$, we get with Young's inequality and $\rho^{-2}|z|^2\leq 1$
    \begin{align*}
       A^\YY(\UU,\VV)
      &\geq \mathfrak{A}_0 \int_{\R^d \times (0,\YY)}{y^\alpha
        \rho^{\varepsilon} \abs{\nabla \UU}^2 dz}
        +
        \varepsilon\int_{\R^d \times (0,\YY)}{y^\alpha
        \rho^{\varepsilon-2} z \cdot  \mathfrak{A}_x\nabla \UU \UU \,dz} \\
      &\geq
        \frac{1}{2}\mathfrak{A}_0
        \int_{\R^d \times (0,\YY)}{y^\alpha\rho^{\varepsilon} \abs{\nabla \UU}^2 dz }
        - \frac{\varepsilon^2}{2}\frac{\norm{\mathfrak{A}_x}_{L^\infty(\R^d \times \R^+)}^2}{\mathfrak{A}_0}
        \int_{\R^d \times (0,\YY)}{y^\alpha
        \rho^{\varepsilon-2} \abs{\UU}^2  dz}\\
      &\geq
       \frac{1}{2}\mathfrak{A}_0 \int_{\R^d \times (0,\YY)}{y^\alpha\rho^{\varepsilon} \abs{\nabla \UU}^2 dz }
        - \frac{C_{P}\varepsilon^2}{2}\frac{\norm{\mathfrak{A}_x}_{L^\infty(\R^d \times \R^+)}^2}{\mathfrak{A}_0}
        \int_{\R^d \times (0,\YY)}{y^\alpha \rho^{\varepsilon} \abs{\nabla \UU}^2 dz },
    \end{align*}
    where in the last step we applied the Poincar\'e estimate from~\eqref{eq:my_poincare} for sufficiently small $\varepsilon >0$. 
    If $\varepsilon$ is sufficiently small, we can also absorb the negative term and show
    inf-sup stability with the test space carrying $\rho^{-\varepsilon}$ as a weight. The non-degeneracy condition and the bound on $(f,\trace \VV)_{L^2(\R^d)}$ are easily checked.
\end{proof}


Before we can proceed to quantify the cutoff error, we need the following result on the
existence of a stable extension from the cutoff domain $\R^d\times (0,\YY)$ to
 a larger set.
\begin{lemma}
  \label{lemma:extension_operator}
  Fix $\YY>0$.
  Then, there exists an extension operator
  $\extension$ to the domain $\R^d \times (0,\frac{3}{2} \YY)$
  such that:
  \begin{enumerate}[(i)]
  \item $\extension u=u$ in $\R^d \times (0,\YY)$.
  
  \item The following stability result holds for all $\mu\geq 0$
    and $\UU \in \HHwry$, if the right-hand side is finite:
    \begin{align}
      \int_{0}^{\frac{3}{2} \YY}{y^{\alpha+\mu} \|\nabla \extension \UU\|_{L^2(\R^d)}^2\,dy}
      &\leq C 
        \int_{0}^{\YY}{y^{\alpha+\mu} \|\nabla \UU\|_{L^2(\R^d)}^2\,dy}.
    \end{align}
    The constant $C>0$ depends on $\alpha$, $\mu$ and $d$  but is independent of
    $\UU$ and $\YY$.
  \end{enumerate}      
\end{lemma}
\begin{proof}
  We extend $\UU$ by reflection along the line $y=\YY$, i.e., we define
  \begin{align*}
    \WW(x,y)&:=\begin{cases}
      \UU(x,y) & 0\leq y \leq \YY, \\
      \UU(x,2\YY-y) & \YY < y\leq \frac{3}{2} \YY.
    \end{cases}
  \end{align*}
  By construction, the function has no jump across the line $y=\YY$.
  For the stability in the extension domain, we compute
  \begin{align*}
    \int_{\YY}^{\frac{3}{2} \YY}{ \! y^{\alpha+\mu} \|\nabla \WW(\cdot,y)\|_{L^2(\R^d)}^2\,dy}
    &\lesssim \YY^{\alpha+\mu} \int_{\YY}^{\frac{3}{2} \YY}{ \|\nabla \UU(\cdot,2\YY-y)\|_{L^2(\R^d)}^2\,dy} \\
    &=\YY^{\alpha+\mu} \int_{\YY/2}^{\YY}{ \|\nabla \UU(\cdot,\tau)\|_{L^2(\R^d)}^2\,d\tau}  \\
    &\lesssim
      \int_{\YY/2}^{\YY}{ \tau^{\alpha+\mu}\|\nabla \UU(\cdot,\tau)\|_{L^2(\R^d)}^2\,d\tau} .
  \end{align*}
  This finishes the proof.
\end{proof}

Using this extension operator, we obtain that the sequence $(\UU^{(3/2)^n\YY})_{n\in\N}$,
where the cutoff point is moved outward by a factor of $3/2$ in each step, is a Cauchy sequence.

\begin{lemma}
  \label{lemma:Cauchy}
  Let $\UU^\YY$ denote the solution to~\eqref{eq:truncated_BLF_eq} with truncation parameter $\YY>0$
  and accordingly let $\UU^{3/2 \YY}$ denote the solution with a cutoff at $3/2\YY$.
  Let $\mu$ be given by \eqref{eq:def_of_decay_mu}.
  Then, there holds:
  \begin{align*}
\|\UU^{3/2\YY} - \UU^\YY\|_{\HH^\YY} &\leq C \YY^{-\mu/2} \norm{f}_{L^2(\Omega)}.
  \end{align*}
  Iterative application of the estimate for $n,m \in \N_0, n>m$ leads to 
    \begin{align*}
      \|\UU^{(3/2)^n\YY} - \UU^{(3/2)^m\YY}\|_{\HH^\YY}
      &\leq C \YY^{-\mu/2} \left(\frac{2}{3}\right)^{\mu\,m/2}\Big(1-\left(\frac{2}{3}\right)^{\frac{\mu}{2}(n-m)}\Big)\norm{f}_{L^2(\Omega)}.
  \end{align*}
\end{lemma}
\begin{proof}
 We compute using the coercivity of $\mathcal{A}^\YY(\cdot,\cdot)$ from Proposition~\ref{prop:discrete_well_posedness} and the extension operator from Lemma~\ref{lemma:extension_operator}
    \begin{align*}
      \|\UU^{\YY} - \UU^{3/2 \YY}\|_{\HH^\YY}^2
      &\lesssim 
      A^{\YY}(\UU^{\YY} - \UU^{3/2 \YY},\UU^{\YY} - \UU^{3/2 \YY}) \\
      &=
      A^{\YY}(\UU^{\YY},\UU^{\YY} - \UU^{3/2 \YY})-
      A^{\YY}(\UU^{3/2\YY},\UU^{\YY} - \UU^{3/2 \YY})\\
      &= (f, \trace(\UU^{\YY} - \UU^{3/2 \YY}))_{L^2(\R^d)} -
      A^{3/2\YY}(\UU^{3/2\YY},\extension(\UU^{\YY} - \UU^{3/2 \YY})) \\
      &\quad
      +\int_{\YY}^{\frac{3}{2}{\YY}}{ y^\alpha \int_{\R^d} \mathfrak{A}_x \nabla \UU^{3/2\YY} \nabla \extension(\UU^{\YY} - \UU^{3/2 \YY}) \,dx dy}.
    \end{align*}
  By definition of $\UU^{3/2\YY}$ and the extension operator $\extension$, the first two terms cancel. 
  Thus, we can focus on bounding the remaining integral
  \begin{multline*}
    \int_{\YY}^{\frac{3}{2}{\YY}}{ y^\alpha\int_{\R^d} \mathfrak{A}_x \nabla \UU^{3/2\YY} \nabla \extension(\UU^{\YY} - \UU^{3/2 \YY}) \; dx dy} \\
    \begin{aligned}[t]
      &\lesssim \YY^{-\mu/2}
      \big(\int_{\YY}^{\frac{3}{2}{\YY}}{ y^{\alpha + \mu} \abs{\nabla \UU^{3/2 \YY}}^2 dy}\big)^{1/2}
      \big(\int_{\YY}^{\frac{3}{2}{\YY}}{ y^{\alpha} \abs{\nabla \extension (\UU^\YY-\UU^{3/2 \YY})}^2 dy}\big)^{1/2} \\
    &\lesssim \YY^{-\mu/2}
      \big(\int_{\YY}^{\frac{3}{2}{\YY}}{ y^{\alpha + \mu} \abs{\nabla \UU^{3/2 \YY}}^2 dy}\big)^{1/2}
      \| \UU^\YY-\UU^{3/2 \YY}\|_{\HHwry}.
    \end{aligned}
  \end{multline*}
  Using $ \|\UU^\YY-\UU^{3/2 \YY}\|_{\HHwry}\leq  \|\UU^\YY-\UU^{3/2 \YY}\|_{\HH^\YY}$
    and canceling one such power then gives
  together with the decay estimate of Lemma~\ref{lemma:stronger_decay}:
  \begin{align} \label{eq:lemmaCauchytmp1}
    \|\UU^{\YY} - \UU^{3/2 \YY}\|_{\HH^\YY}
    &\lesssim \YY^{-\mu/2}\norm{f}_{L^2(\Omega)}.      
  \end{align}
  Using a telescoping sum, we can write:
  \begin{align*}
    \UU^{(3/2)^n\YY} - \UU^{(3/2)^m\YY} =
    \sum_{\ell=m}^{n-1}{ \Big(\UU^{(3/2)^{\ell+1}\YY} - \UU^{(3/2)^{\ell}\YY} \Big)}.
  \end{align*}
With estimate \eqref{eq:lemmaCauchytmp1} applied iteratively, this leads to
  \begin{align*}
    \|\UU^{(3/2)^n\YY} - \UU^{(3/2)^m\YY} \|_{\HH^\YY}
    &\lesssim
      \sum_{\ell=m}^{n-1}{ \| \UU^{(3/2)^{\ell+1}\YY} - \UU^{(3/2)^{\ell}\YY}\|_{\HH^\YY}
      } 
    \lesssim \YY^{-\mu/2} \sum_{\ell=m}^{n-1}{\Big(\frac{3}{2} \Big)^{-\frac{\mu \ell}{2}}}\norm{f}_{L^2(\Omega)} \\
      &\simeq \YY^{-\mu/2} \left(\frac{2}{3}\right)^{\frac{\mu}{2}m}\Big(1-\left(\frac{2}{3}\right)^{\frac{\mu}{2}(n-m)}\Big) \norm{f}_{L^2(\Omega)}. 
  \end{align*}
  This finishes the proof. 
\end{proof}

Using the Cauchy sequence property, we can now show convergence of the truncated solution to the 
full-space solution as stated in Proposition~\ref{lemma:truncation_convergence}.

\begin{proof}[Proof of Proposition~\ref{lemma:truncation_convergence}]
  We focus on the case $s=0$. In the case $s>0$, the same arguments can be made including the $L^2$-norm of of the traces, which directly gives the additional statement regarding the convergence of $\trace \UU^\YY$ to $\trace \UU$.

\textbf{Step 1:} We start by fixing the half-ball $B_\YY^+ \subset \R^d \times [0,\infty)$ of radius $\YY$ centered at the origin and write 
$z = (x,y) \in \R^{d+1}$. 
Let $\varepsilon>0$ be such that the decay estimate~\eqref{eq:decay_all_directions} holds.

Defining $E:= \UU-\UU^\YY$ and using the equations satisfied by $\UU$ and $\UU^\YY$, we use integration by parts to obtain
\begin{align*}
\int_{B_\YY^+} y^\alpha \mathfrak{A}_x \nabla E \cdot \nabla E \; dx dy &= \int_{\partial B_\YY^+} y^\alpha  \mathfrak{A}_x  \nabla E \cdot \nu E \; dx dy \\
 &= (1+\YY^2)^{-\varepsilon/2}\int_{\abs{z}=\YY} y^\alpha  \rho^{\varepsilon} \mathfrak{A}_x  \nabla E \cdot \nu E \; dx dy - s d_\beta \int_{\abs{x} \leq \YY} |\trace E|^2 \; dx \\
 &= (1+\YY^2)^{-\varepsilon/2}\int_{\partial B_\YY^+} y^\alpha  \rho^{\varepsilon}\mathfrak{A}_x  \nabla E \cdot \nu E \; dx dy  \\ &\qquad +sd_\beta\int_{\abs{x}\leq \YY} \Big(\frac{1+|x|^2}{1+\YY^2}\Big)^{\varepsilon/2} |\trace E|^2 \; dx - sd_\beta \int_{\abs{x} \leq \YY} |\trace E|^2 \; dx \\
 &\leq (1+\YY^2)^{-\varepsilon/2}\int_{\partial B_\YY^+} y^\alpha  \rho^{\varepsilon}\mathfrak{A}_x  \nabla E \cdot \nu E \; dx dy.
\end{align*}
Integration by parts back (replacing $\nabla E$ by $\nabla (\rho^\varepsilon  E)$) gives
\begin{align*}
\int_{\partial B_\YY^+} y^\alpha  \rho^{\varepsilon}\mathfrak{A}_x  \nabla E \cdot \nu E \; dx dy &= 
\int_{B_\YY^+} y^\alpha \mathfrak{A}_x  \nabla E  \cdot (\nabla \rho^\varepsilon) E \; dx dy + \int_{B_\YY^+} y^\alpha  \rho^{\varepsilon}\mathfrak{A}_x \nabla E\cdot \nabla E \; dx dy \\
& \lesssim\Big(\int_{B_\YY^+}{y^\alpha \rho^{\varepsilon} \abs{\nabla  E}^2 \,dz} \Big)^{1/2}
    \Big(\int_{B_\YY^+}{y^\alpha \rho^{\varepsilon-2} \abs{ E}^2 \,dz} \Big)^{1/2} \\
    &\quad +\int_{B_\YY^+}{y^{\alpha}\rho^{\varepsilon} |\nabla E|^2\,dz}.
\end{align*}
We replace the half-ball $B_\YY^+$ by the cylinder $\R^d \times (0,\YY)$ and use the
Poincar\'e estimate \eqref{eq:my_poincare}. Together with the decay estimate~\eqref{eq:decay_all_directions} this gives boundedness of the right-hand side with a constant independent of $\YY$. Consequently, we obtain
\begin{align*}
\int_{B_R^+} \abs{\nabla E}^2 \; dx dy \lesssim \int_{B_\YY^+} y^\alpha \mathfrak{A}_x \nabla E \cdot \nabla E \; dx dy \leq C (1+\YY^2)^{-\varepsilon/2} \rightarrow 0 \qquad \text{ as } \YY \rightarrow  \infty
\end{align*}
for all bounded half balls $B_R^+$ with $R \leq \YY$, which gives $\UU^\YY \rightarrow \UU$ in $H^1_\rho(y^\alpha,B_R^+)$.

\textbf{Step 2:} As  $(\UU^{(3/2)^n\YY})_{n\in\N}$ is a Cauchy-sequence, there exists a limit $\widehat \UU \in H^1_{\rho}(y^\alpha, \R^d \times (0,\widetilde \YY))$. Assume that $\widehat \UU \neq \UU$. Then, there has to exist a half ball $B_R^+$ such that $\displaystyle\int_{B_R^+} y^\alpha \Big|\nabla(\UU-\widehat \UU)\Big|^2 \; dx dy \neq 0$. For sufficiently large $n$, we have $R \leq (3/2)^n \YY$.
This leads to
\begin{align*}
 \int_{B_R^+} y^\alpha \Big|\nabla(\UU-\widehat \UU)\Big|^2 \; dx dy \leq 
 \int_{B_R^+} y^\alpha \Big|\nabla(\UU-\UU^{(3/2)^n\YY})\Big|^2  dx dy + \int_{B_R^+} y^\alpha \Big|\nabla(\UU^{(3/2)^n\YY}-\widehat \UU)\Big|^2  dx dy.
\end{align*}
By step~1, the first term converges to zero and by definition of $\widehat \UU$ the second term converges to zero. However, this is a contradiction to the assumption and therefore $\UU=\widehat \UU$ and we have established the claimed convergence.
\end{proof}

We can now estimate the truncation error and establish a rate of convergence as $\YY \rightarrow \infty$.

\begin{proof}[Proof of Proposition~\ref{lemma:truncation_error}]
Using a telescoping sum, we write
  \begin{align*}
    \UU^{\YY} - \UU =
    \sum_{n=0}^{N}{ \Big(\UU^{\YY (\frac{3}{2})^n} - \UU^{\YY (\frac{3}{2})^{n+1}} \Big)}
    + \UU^{\YY (\frac{3}{2})^{N+1}}- \UU.
  \end{align*}
  Since we have already established that $\UU^\YY \to \UU$ for $\YY\to \infty$ in Proposition~\ref{lemma:truncation_convergence},  we can pass to the limit $N \to \infty$ and use Lemma~\ref{lemma:Cauchy} to estimate:
  \begin{align*}
    \|\UU^{\YY} - \UU\|_{\HHwry}
    &\lesssim
      \sum_{n=0}^{\infty}{ \| \UU^{\YY (\frac{3}{2})^n} - \UU^{\YY (\frac{3}{2})^{n+1}}\|_{\HHwry}} \\
    &\lesssim \YY^{-\mu/2} \sum_{n=0}^{\infty}{\Big(\frac{3}{2} \Big)^{-\frac{\mu n}{2}}}\norm{f}_{L^2(\R^d)}
      \leq \YY^{-\mu/2} \frac{1}{1- (\frac{2}{3})^{\mu/2}} \norm{f}_{L^2(\R^d)}. 
  \end{align*}
  This finishes the proof.
\end{proof}

  We can now also close the small gap that the decay in Lemma~\ref{lemma:stronger_decay}
  does not hold for the non-truncated domain $\YY=\infty$.
  
\begin{corollary}
  \label{cor:decayfull}
 Let $\mu$ be given by \eqref{eq:def_of_decay_mu}.
    Let $\UU$ solve~\eqref{eq:weakform}.
    Then, there exists a constant $C>0$ depending only on $\Omega, d$, and $\alpha$ such that
    \begin{align}
      \label{eq:decay_in_y}
      \int_{0}^{\infty}{y^{\alpha}\Big[(1+y)^\mu \|\nabla \UU(y)\|_{L^2(\R^d)}^2
      +(1+y)^\mu\|\rho(\cdot,y)^{-1} \UU(y)\|_{L^2(\R^d)}^2\Big]\,dy}
      \leq C \norm{f}_{L^2(\Omega)}^2.
    \end{align}
\end{corollary}
  \begin{proof}
   We take a sequence $(\YY_n)_{n \in \N}$ with $1\leq \YY_n \to \infty$ for $n \rightarrow \infty$ and
    consider the corresponding truncated solutions $\UU^{\YY_n}$ to \eqref{eq:truncated_BLF_eq}.
    By Lemma~\ref{lemma:stronger_decay} and Proposition~\eqref{lemma:truncation_error} it holds:
    \begin{multline*}
      \int_{0}^{\YY_n}{y^{\alpha}(1+y)^\mu \|\nabla \UU(y)\|_{L^2(\R^d)}^2\,dy} + \int_{0}^{\YY_n}{y^{\alpha}(1+y)^\mu \|\rho^{-1}\UU(y)\|_{L^2(\R^d)}^2\,dy} \\
      \begin{aligned}[t]
        &\leq
        (1+\YY_n)^{\mu}\norm{\UU-\UU^{\YY_n}}_{H^1_\rho(y^\alpha,\R^d \times(0,\YY_n))}^2 \\
       &\quad +\int_{0}^{\YY_n}{y^{\alpha}(1+y)^\mu \|\nabla \UU^{\YY_n}(y)\|_{L^2(\R^d)}^2 \,dy}+\int_{0}^{\YY_n}{y^{\alpha}(1+y)^\mu \|\rho^{-1} \UU^{\YY_n}(y)\|_{L^2(\R^d)}^2 \,dy}\\
        &\lesssim \YY_n^{\mu}\YY_n^{-\mu} \|f\|_{L^2(\Omega)}^2 +  \min(s^{-1},1)^2\norm{f}_{L^2(\Omega)}^2
        \lesssim  \|f\|_{L^2(\Omega)}^2.
      \end{aligned}
    \end{multline*}

    Taking $n \to \infty$ then gives the stated result.
  \end{proof}

\section{Regularity and higher order decay}
\label{sec:regularity}
In this section, we derive regularity estimates for solutions to the extension problem. 
Assuming sufficient differentiability of the data, we are in particular interested in weighted estimates for higher-order $y$-derivatives as such estimates are needed to establish exponential approximation estimates of $hp$--type. 

In order to derive suitable regularity estimates around $y=0$, we need to derive an initial shift in a weighted space.
\begin{lemma}\label{lemma:initial_shift}
  Fix $\YY \in (0,\infty]$.
  Let $\UU$ solve \eqref{eq:truncated_BLF_eq}. 
Then, there exists $\varepsilon > 0$ independent of $\YY$ and $\UU$ such that
\begin{align}
    \int_{0}^{\YY}{y^{\alpha}\Big[y^{-\varepsilon} \|\nabla \UU(y)\|_{L^2(\R^d)}^2
    +y^{-\varepsilon}\|\rho(\cdot,y)^{-1} \UU(y)\|_{L^2(\R^d)}^2\Big]\,dy}
    \leq C  \norm{f}_{L^2(\Omega)}^2. 
\end{align}
\end{lemma}

\begin{proof}
Similar to the proof of Lemma~\ref{lemma:stronger_decay}, we use inf-sup theory to derive the stated bound. 
In the following, we only work out the details for the case $s=0$. The case $s>0$ can be treated as shown in Lemma~\ref{lemma:stronger_decay} by also including a trace term in the norm of the ansatz space.

Here,  for any $\widetilde{\varepsilon} \in \R$, we define the space
$\mathcal{X}_{\widetilde{\varepsilon} ,\YY}$ as the space $H_{\rho}^1(y^{\alpha-\widetilde{\varepsilon}},\R^d\times (0,\YY))$ of functions with finite norm
  $$
  \|\UU\|_{\mathcal{X}_{\widetilde{\varepsilon} ,\YY}}^2:= \int_{0}^{\YY}{y^{\alpha-\widetilde{\varepsilon} }\Big[\|\nabla \UU(y)\|_{L^2(\R^d)}^2 
    +\| \rho(\cdot,y)^{-1} \UU(y)\|_{L^2(\R^d)}^2\Big]\,dy }.
     $$

As ansatz space, we take $\mathcal{X}_{\varepsilon,\YY}$, where $\varepsilon>0$ is sufficiently small.
As test space we use $\mathcal{X}_{-\varepsilon,\YY}$. For fixed $\alpha  \in (-1,1)$, we actually may choose $\varepsilon > 0$ such that $\alpha \pm \varepsilon \in (-1,1)$ (subsequently, we will derive an additional restriction on $\varepsilon$).  

For given $\UU \in \mathcal{X}_{\varepsilon,\YY}$, we define the test function
$\VV(x,y) := y^{-\varepsilon} \UU(x,y) + \varepsilon \int_0^y \tau^{-\varepsilon-1} \UU(x,\tau) d\tau$. 

Using Hardy's inequality (noting that $\alpha+\varepsilon>-1$), we obtain that this test-function is indeed in the test-space
\begin{align}\label{eq:tmpshifty1}
\int_0^\YY y^{\alpha + \varepsilon} \norm{\nabla \VV(y)}_{L^2(\R^d)}^2  dy &\lesssim \int_0^\YY y^{\alpha+\varepsilon}y^{-2\varepsilon} \norm{\nabla \UU(y)}_{L^2(\R^d)}^2  dy \nonumber \\
&\qquad + \int_{\R^d}\int_0^\YY y^{\alpha+\varepsilon} \left(\varepsilon\int_0^y \tau^{-\varepsilon-1} \nabla_x \UU (\tau) d\tau\right)^2 dy dx \nonumber \\
&\lesssim (1+\varepsilon^2)\int_0^\YY y^{\alpha-\varepsilon}\norm{\nabla \UU(y)}_{L^2(\R^d)}^2 dy < \infty.
\end{align}
The weighted $L^2$-term in the definition of $\mathcal{X}_{\varepsilon,\YY}$ can be treated using the Poincar\'e inequality \eqref{eq:my_poincare} replacing $\alpha$ with $\alpha - \varepsilon$ therein noting that $\alpha - \varepsilon \in (-1,1)$ by assumption on $\varepsilon$.

Inserting the test function into the bilinear form gives 
\begin{align*}
A^\YY(\UU,\VV) &= \int_{\R^d} \int_0^\YY y^{\alpha-\varepsilon} \mathfrak{A}_x \nabla \UU \cdot \nabla \UU dy dx + \varepsilon \int_{\R^d}\int_0^\YY y^\alpha\mathfrak{A} \nabla_x \UU \int_0^y \tau^{-\varepsilon-1} \nabla_x \UU(\tau) d\tau\, dydx. 
\end{align*}
Using Young's inequality together with Hardy's inequality (noting again that $\alpha + \varepsilon >-1$), we obtain
\begin{align*}
\varepsilon\int_{\R^d}\int_0^\YY y^\alpha\mathfrak{A} \nabla_x \UU \int_0^y \tau^{-\varepsilon-1} \nabla_x \UU(\tau) d\tau\, dydx &\leq \frac{1}{2}\int_{\R^d} \int_0^\YY y^{\alpha-\varepsilon} \mathfrak{A}_x \nabla \UU \cdot \nabla \UU dy dx \\
&\quad + \frac{1}{2}\varepsilon^2 \int_{\R^d} \int_0^\YY y^{\alpha+\varepsilon} \left(\int_0^y \tau^{-\varepsilon-1}\mathfrak{A}^{1/2}\nabla_x \UU(\tau) d\tau\right)^2 dydx \\
&\leq 
\frac{1}{2}\left(1+C_{H}\varepsilon^2\right)\int_{\R^d} \int_0^\YY y^{\alpha-\varepsilon} \mathfrak{A}_x \nabla \UU \cdot \nabla \UU \; dy dx,
\end{align*}
where $C_H$ indicates the constant in the Hardy inequality. Therefore, we obtain
\begin{align*}
A^\YY(\UU,\VV) \geq \frac{\mathfrak{A}_0}{2} \left(1 - C_H\varepsilon^2\right)  \int_0^\YY y^{\alpha-\varepsilon} \norm{\nabla \UU}_{L^2(\R^d)}^2 dy.
\end{align*}
Together with the Poincar\'e estimate of Lemma~\ref{lemma:my_poincare}, we obtain the inf-sup condition upon choosing $\varepsilon < C_H^{-1/2}$.

For the non-degeneracy condition, we fix $\VV \in \mathcal{X}_{-\varepsilon,\YY}$ and choose $\UU = y^\varepsilon \VV - \varepsilon \int_0^y \tau^{\varepsilon-1} \VV(\tau) d\tau$. Then, essentially the same estimates as above can be made by noting that, by assumption we have $\alpha-\varepsilon >-1$, thus Hardy inequalities with the necessary modified weights can be employed here. 

The  right-hand side can be bounded using the support properties of $f$ together with a trace estimate (in the weighted space $L^2(y^{\alpha+\varepsilon},\Omega\times (0,\YY)) $ noting that $\alpha + \varepsilon \in (-1,1)$) 
$$
\abs{(f,\trace\VV)_{L^2(\R^d)}} \leq \norm{f}_{L^2(\Omega)} \norm{\trace \VV}_{L^2(\Omega)} \leq \norm{f}_{L^2(\Omega)} \norm{\nabla \VV}_{L^2(y^{\alpha+\varepsilon},\Omega\times (0,\YY))} \leq \norm{f}_{L^2(\Omega)} \norm{\VV}_{\mathcal{X}_{-\varepsilon,\YY}}.
$$
Now, classical inf-sup theory gives the claimed estimate.
\end{proof}

With the initial shift in place, we can look at higher order derivatives. We first
formulate the ``shift-by-one'' as a separate lemma.

\begin{lemma}
  \label{lemma:apriori_decay_l2}
  Fix $\YY \in (0,\infty]$ and let $\WW \in H^1_\rho(y^\alpha,\R^d \times (0,\YY))$ solve the problem
  \begin{align*}
    -\fdiv\big(y^\alpha \mathfrak{A}_x \nabla \WW\big) &= F \qquad \text{in $\R^d\times (0,\YY)$}
  \end{align*}
with given right-hand side $F$.
  Then, for all $\ell \in \N$ and $\varepsilon \in (0,1)$, the estimate
  \begin{align*}
    \big\|{y^{\ell-\varepsilon}\nabla \WW}\big\|_{L^2(y^\alpha,\R^d \times (0,\YY))}
    &\lesssim  \ell\big\|{y^{\ell-1-\varepsilon}\WW}\big\|_{L^2(y^\alpha,\R^d \times (0,\YY))}
      + \big\|{y^{\ell+1-\varepsilon}F}\big\|_{L^2(y^{-\alpha},\R^d \times (0,\YY))}
  \end{align*}
  holds, provided that the right-hand side is finite.  The implied
  constant is independent of $\ell$ and $\WW$.
\end{lemma}
\begin{proof}
  If $\YY=\infty$, let $N \in \N$, and
  we fix a cutoff function $\tilde{\chi}_N \in C_{0}^{\infty}(\R)$ such that $\tilde{\chi}_N \equiv 1$ on $[0,N]$ and
  $\tilde{\chi}_N \equiv 0$ on $(2N,\infty)$ with $\|\tilde{\chi}_N'\|_{L^\infty(\R)} \leq 1/N$.
  We define $\omega_N(y):=y^{\ell-\varepsilon} \tilde{\chi}_N$. In the easier case $\YY<\infty$, we can skip the cutoff
  function altogether. For brevity, we therefore only work out the case $\YY=\infty$,  the other case follows analogously. 
  
  We start with multiplying the equation for $\WW$ with the test function $\VV:=\omega_N^2 \WW$, and integrate by parts 
  over $\R^d \times (0,\infty)$. 
  As the weight function $\omega_N$ and consequently also $\VV$ vanishes at $y=0$, we do not get any boundary contributions.
  This gives with Young's inequality
  \begin{multline*}
    \|\omega_N \mathfrak{A}_x^{1/2} \nabla \WW \|_{L^2(y^\alpha,\R^d\times \R^+)}^2 \\
    \begin{aligned}[t]
      &=\int_{\R^d \times \R^+}{ \omega_N^2(y)F\,\WW dx dy}
        - \int_{\R^d \times \R^+}{ 2 \omega_N'(y)\omega(y) \partial_y \WW \WW dxdy} \\
      &\leq \|y \omega_N F\|_{L^2(y^{-\alpha},\R^d\times \R^+)} \|y^{-1}\omega_N \WW\|_{L^2(y^{\alpha},\R^d\times \R^+)}
        + 2\|\omega_N \partial_y \WW\|_{L^2(y^\alpha,\R^d\times \R^+)}
        \|\omega_N'  \WW\|_{L^2(y^\alpha,\R^d\times \R^+)} \\
      &\leq
        \frac{1}{2}\|y \omega_N F\|^2_{L^2(y^{-\alpha},\R^d\times \R^+)} +  \frac{1}{2}\|y^{-1}\omega_N \WW\|^2_{L^2(y^{\alpha},\R^d\times \R^+)} \\
      &\quad  + \frac{1}{2}\|\omega_N \partial_y \WW\|_{L^2(y^\alpha,\R^d\times \R^+)}^2
        +2\|\omega'_N  \WW\|_{L^2(y^\alpha,\R^d\times \R^+)}^2.
      \end{aligned}
    \end{multline*}
 
 Absorbing the third term in the left-hand side provides
 \begin{align*}
   \|\omega_N \mathfrak{A}_x^{1/2} \nabla \WW \|_{L^2(y^\alpha,\R^d\times \R^+)}^2
   &\lesssim
   \|y\omega_N F\|^2_{L^2(y^{-\alpha},\R^d\times \R^+)} +  \|y^{-1}\omega_N \WW\|^2_{L^2(y^{\alpha},\R^d\times \R^+)}    \\
   &\quad+\|\omega'_N  \WW\|_{L^2(y^\alpha,\R^d\times \R_+)}^2.        
 \end{align*}
 For $N \to \infty$, using $\mathfrak{A}_x \geq \mathfrak{A}_0$, the left-hand side converges to the weighted $L^2$-norm we are looking for.
 Similarly, the first two terms on the right-hand side converge to the appropriate objects
 of the final estimate. Therefore we focus on the last term and show an uniform bound:
 \begin{align*}
   \|\omega'_N  \WW\|_{L^2(y^\alpha,\R^d\times \R^+)}
   &\leq (\ell-\varepsilon)\|y^{\ell-1-\varepsilon} \tilde{\chi}_N \WW\|_{L^2(y^\alpha,\R^d\times \R^+)}
     +\|y^{\ell-\varepsilon} \tilde{\chi}'_N \WW\|_{L^2(y^\alpha,\R^d\times \R^+)} \\
   &\lesssim \ell \|y^{\ell-1-\varepsilon}\WW\|_{L^2(y^\alpha,\R^d\times \R^+)}
     + \frac{1}{N^2}\int_{N}^{2N}{\underbrace{y^2}_{\lesssim 4N^2} y^{\alpha + 2\ell -2  -2\varepsilon} \|\WW(y)\|^2_{L^2(\R^d)}\,dy}\\
   &\lesssim \ell\|y^{\ell-1-\varepsilon}\WW\|_{L^2(y^\alpha,\R^d\times \R^+)}
     + \int_{0}^{\infty}{ y^{\alpha +2\ell -2 -2\varepsilon} \|\WW(y)\|^2_{L^2(\R^d)}\,dy},
 \end{align*}
 where we used that $\tilde{\chi}'_N$ vanishes outside of $[N,2N]$. Therefore we can pass to
 the limit $N\rightarrow \infty$ to get the stated result. 
\end{proof}

\begin{remark}
  Note that $\UU$ as solution of \eqref{eq:extension} does not fit Lemma~\ref{lemma:apriori_decay_l2} since it is not in $L^2_{\alpha}(\R^d\times (0,\YY))$. However, the previous lemma can be applied for derivatives of the solution of \eqref{eq:extension}.
\end{remark}

We are now in position to show our main result regarding weighted regularity, Proposition~\ref{prop:regularity}.

\begin{proof}[Proof of Proposition~\ref{prop:regularity}]
  We note that away from $y=0$, we can use standard elliptic regularity theory to show that
  $\UU$ is $C^{\infty}(\R^d\times \R)$ and we can focus on the weighted estimates.
  We prove this by induction, starting with $\ell=1$.
  By differentiating the equation in the form 
  $\fdiv(\mathfrak{A}_x\nabla \UU) + \frac{\alpha}{y} \partial_y \UU = 0$,
  we get that $\WW:=\partial^{\ell}_y \UU$ solves:
  \begin{align}\label{eq:CSextension_diff_y}
  -\fdiv(y^\alpha \mathfrak{A}_x \nabla \WW) = \alpha\sum_{k=0}^{\ell-1}{(-1)^{k}\frac{\ell!}{k!} \frac{ \partial^{k+1}_y\UU}{y^{\ell-k+1-\alpha}}}
  =:F_{\ell}.
  \end{align}
  For $\ell = 1$, we employ Lemma~\ref{lemma:apriori_decay_l2} to obtain
  \begin{align*}
    \norm{y^{1-\varepsilon}\nabla \partial_y \UU}_{L^2(y^{\alpha},\R^d\times(0,\YY))}
    &\lesssim  \norm{y^{-\varepsilon}\partial_y \UU}_{L^2(y^{\alpha},\R^d\times (0,\YY))} +
       \|{y^{2-\varepsilon} y^{-2+\alpha}\partial_y \UU}\|_{L^2(y^{-\alpha},\R^d\times (0,\YY))} \\
   &\lesssim  \norm{y^{-\varepsilon} \partial_y \UU}_{L^2(y^\alpha,\R^d\times (0,\YY))} \lesssim \norm{f}_{L^2(\Omega)},
  \end{align*}
  where in the last step we used Lemma~\ref{lemma:initial_shift}.

  For $\ell>1$, we use the induction assumption valid for $k<\ell$ (that allows to control derivatives up to order $\ell$), which gives
  \begin{align*}
    \big\|{y^{\ell+1-\varepsilon}F_\ell}\big\|_{L^2(y^{-\alpha},\R^d\times(0,\YY))}
    &\lesssim \sum_{k=0}^{\ell-1}\frac{\ell!}{k!}
      \big\|{y^{k-\varepsilon}\partial_y^{k+1} \UU}\big\|_{L^2(y^\alpha,\R^d\times (0,\YY))} \\
   &\lesssim   \ell!\norm{f}_{L^2(\R^d)} \sum_{k=0}^{\ell-1} K^k \\
    &\lesssim \ell! K^\ell\norm{f}_{L^2(\R^d)}.
  \end{align*}
  Using Lemma~\ref{lemma:apriori_decay_l2} together with the induction assumption, we get
  \begin{align*}
    \big\|y^{\ell-\varepsilon} {\nabla \partial_y^\ell \UU}\big\|_{L^2(y^{\alpha},\R^d\times(0,\YY))}
    &\lesssim    \ell \big\|{y^{\ell-1-\varepsilon} \partial_{y}^{\ell} \UU}\big\|_{L^2(y^{\alpha},\R^d\times(0,\YY))} +    \big\|{y^{\ell+1-\varepsilon}F_\ell}\big\|_{L^2(y^{-\alpha},\R^d\times(0,\YY))} \\
    &\lesssim    
      \ell \big\|{y^{\ell-1-\varepsilon} \nabla \partial_{y}^{\ell-1} \UU}\big\|_{L^2(y^{\alpha},\R^d\times(0,\YY))} + \ell!K^\ell\big\|{f}\big\|_{L^2(\R^d)} \\
      &\lesssim \ell!K^{\ell} \norm{f}_{L^2(\R^d)},
  \end{align*}
  which proves the lemma.
\end{proof}

Finally, we provide the proof for the regularity estimates for the $x$-derivatives.

\begin{proof}[Proof of Proposition~\ref{prop:regularity_x}]
In order to obtain estimates for the $x$-derivatives, for a given multi-index $\zeta$, we differentiate the equation with respect to $\partial_x^\zeta$. As the weight $y^\alpha$ remains unchanged, we see that $\WW:=\partial_{x}^\zeta \UU$ solves the extension problem \eqref{eq:extension}
\begin{align*}
 - \fdiv\big(y^\alpha \mathfrak{A}_x \nabla \WW\big) &= F_\zeta \qquad \text{in $\R^{d} \times \R^+$},\\
  d^{-1}_{\beta}\partial_{\nu^\alpha} \WW +s \trace\WW&= f_\zeta \qquad \text{in $\R^{d}$},
\end{align*}
 with data $f_\zeta := \partial_{x}^\zeta f$ and right-hand side
$$
F_\zeta := -\fdiv\Big(y^\alpha\sum_{\zeta' < \zeta} \binom{\zeta}{\zeta'} (\partial_x^{\zeta-\zeta'} \mathfrak{A}_x) \partial_x^{\zeta'} \nabla \UU\Big).
$$
One can modify the arguments of Proposition~\ref{prop:discrete_well_posedness} to also include the source term $(F_\zeta,\WW)_{L^2(\R^d\times\R^+)}$, which can be estimated using 
\begin{align*}
\abs{(F_\zeta,\WW)_{L^2(\R^d\times\R^+)}} \leq \norm{F_\zeta}_{L^2(y^{-\alpha},\R^d\times\R^+)} \norm{\WW}_{L^2(y^{\alpha},\R^d\times\R^+)}.
\end{align*}
This gives 
\begin{align*}
\norm{\nabla \WW}_{L^2(y^\alpha, \R^d\times\R^+)} \lesssim \norm{F_\zeta}_{L^2(y^{-\alpha},\R^d\times\R^+)} + \norm{f_\zeta}_{L^2(\Omega)}.
\end{align*}
Now, an induction argument can be set up as in the proof of Proposition~\ref{prop:regularity} to control $\norm{F_\zeta}_{L^2(y^{-\alpha},\R^d\times\R^+)}$ by $L^2$-norms of derivatives of $f$.
\end{proof}

\bibliographystyle{alphaabbr}
\bibliography{literature}

\end{document}